\documentclass{amsart}
\usepackage{amsmath,amsfonts,amssymb,amsthm, hyperref, listings, xcolor} 

\title[Multivariate Hermite Method]{The multivariate Hermite method for counting real and complex solutions to polynomial systems}
\author{Volodymyr Oleksiyuk}

\lstdefinelanguage{Macaulay2}{
  morekeywords={
    if, then, else, for, do, while, return, new, true, false,
    ideal, matrix, ring, ZZ, QQ, R, S, apply, vars, entries, flatten,
    table, print, sort, concatenate, value, degree, max, coefficients, trace, eigenvalues, sum, terms, leadCoefficient, length, mutableMatrix, substitute, timing, isNumber, not, error, toString, toList, dim, monomials, set, rank, signature
  },
  sensitive=true,
  morecomment=[l]--,
  morestring=[b]",
}

\begin{document}

\newtheorem{lemma}{Lemma}
\theoremstyle{definition}
\newtheorem*{theorem}{Theorem}
\begin{abstract}

This note presents the multivariate Hermite criterion: a practical and powerful algorithm for determining the number of distinct real and complex roots of a zero-dimensional system of polynomials in any finite number of variables. The final section includes an implementation in Macaulay2, a free and open-source computer algebra system.

\end{abstract}

\maketitle

\vspace{2em}

\begin{center}
    \textit{Written under the guidance of Prof. Dr. Daniel Plaumann}
\end{center}

\vspace{2em}

\section{Introduction}

Linear (degree 1) and quadratic (degree 2) polynomials are solvable in closed form. Explicit solutions also exist for cubic and quartic equations via the Cardano formula and the Ferrari method, respectively. The Abel–Ruffini theorem shows that no general radical solution exists for polynomial equations of degree five or higher. Numerical methods such as Newton’s method can approximate roots but do not determine their exact number. While a univariate polynomial of degree $n$ has $n$ complex roots counted with multiplicity, identifying the number of distinct roots, and in particular, the number of real roots, requires additional tools. The Hermite method provides these counts.

The result addressed here is classical; the objective of this research paper is to present a proof of the multivariate Hermite criterion that closely parallels the univariate case, as well as an implementation for this method. 

\vspace{1em}

\section{The Hermite Criterion and Its Generalization}

\vspace{1em}

\subsection{Classic Hermite Method}~

The classical Hermite method for a single polynomial proceeds as follows:

\begin{theorem}[Classic Hermite Criterion]
Let $f = t^n + a_{n-1} t^{n-1} + \ldots + a_1 t + a_0 \in \mathbb{R}[t]$. Let $\alpha_1, \ldots, \alpha_n \in \mathbb{C}$ be the (not necessarily distinct) roots of $f$. Define the Newton sums as $p_k = \alpha_1^k + \alpha_2^k + \ldots + \alpha_n^k$. Then, for the Hermite matrix $H(f) = (p_{i+j-1})_{i,j = 1,\ldots,n}$:
\begin{enumerate}
    \item The rank of $H(f)$ equals the number of distinct complex roots of $f$.
    \item The signature of $H(f)$ equals the number of distinct real roots.
\end{enumerate}
\end{theorem}

At first glance, this theorem might appear redundant, since the Newton sums $p_k$ are defined in terms of the roots of $f$. If the roots are already known, why compute the matrix? The key insight is that the Newton identities express the $p_k$ directly in terms of the coefficients of $f$:
\[
p_r + a_{n-1} p_{r-1} + a_{n-2} p_{r-2} + \ldots + a_{n-r+1} p_1 + a_{n-r} r = 0, \quad \forall r \ge 1
\]

This allows recursive computation of the $p_k$:
\[
p_0 = n,\quad p_1 = -a_{n-1},\quad p_2 = -p_1 a_{n-1} - 2a_{n-2} = a_{n-1}^2 - 2a_{n-2},\quad \text{etc.}
\]

A proof of this classical result can be found in \cite[Theorem~4.13, p.~99]{basu_pollack_roy2003book}. Here, we instead focus on the generalized Hermite criterion for systems of polynomials in an arbitrary number of variables.

\vspace{1em}

\subsection{Multivariate Hermite criterion}~

\begin{theorem}[Multivariate Hermite criterion, \cite{plaumann2023seminar}]
Let $f_1, \ldots, f_k \in \mathbb{R}[X_1, \ldots, X_n]$ and assume that the system $f_1 = \cdots = f_k = 0$ is zero-dimensional, i.e.~ with finitely many complex solutions. Let $I = (f_1, \ldots, f_k)$ be the ideal generated by $f_1, \ldots, f_k$ and $A = \mathbb{R}[X_1, \ldots, X_n]/I$ the factor ring modulo $I$. On the finite-dimensional $\mathbb{R}$-vector space $A$, let
\[
\varphi_f : A \to A, \, h \mapsto fh
\]
be the multiplication with a fixed element $f \in A$, and consider the bilinear map
\[
H : A \times A \to \mathbb{R}, \, (f, g) \mapsto \operatorname{tr}(\varphi_{fg}).
\]
\begin{enumerate}
    \item The rank of $H$ is the number of distinct complex solutions.
    \item The signature of $H$ is the number of distinct real solutions.
\end{enumerate}
\end{theorem}

\vspace{1em}

\subsection{Proof under an assumption}~

We first prove the theorem under the assumption that $A$ is the ring of all functions $g : \mathfrak{V}(I) \to \mathbb{C}$, that are induced by a real polynomial
\[
A = \{g : \mathfrak{V}(I) \to \mathbb{C} \mid g \in \mathbb{R}[X_1, \ldots, X_n] \},
\]
where $\mathfrak{V}(I) = \{x_1, \ldots, x_r, z_1, \ldots, z_s, \overline{z_1}, \ldots, \overline{z_s} \}$ is the finite set of distinct complex solutions to the system $f_1 = \cdots = f_k = 0$. Here, $\{x_1, \ldots, x_r\}$ are the real roots, while the nonreal ones are grouped by conjugate pairs $\{z_1, \ldots, z_s\}$ and $\{\overline{z_1}, \ldots, \overline{z_s}\}$. 

We construct a basis of $A$ using the following functions:
\[
\varepsilon_l \in \mathbb{R}[X_1, \ldots, X_n], \quad l \in \{1, \ldots, r+2s\},
\]
defined on $\mathfrak{V}(I)$ by:
\begin{align*}
    \varepsilon_l(x_j) &= \delta_{lj},\quad \varepsilon_l(z_j) = 0, & l \in \{1, \ldots, r\}, \\
    \varepsilon_l(x_j) &= 0,\quad \varepsilon_l(z_j) = \delta_{l j}, & l \in \{r+1, \ldots, r+s\}, \\
    \varepsilon_l(x_j) &= 0,\quad \varepsilon_l(z_j) = \delta_{l j} i, & l \in \{r+s+1, \ldots, r+2s\}.
\end{align*}
Since these $\varepsilon_l$ are real polynomials, their values at the conjugates satisfy:
\begin{align*}
    \varepsilon_l(\overline{z_j}) &= 0, & l \in \{1, \ldots, r\}, \\
    \varepsilon_l(\overline{z_j}) &= \delta_{l j}, & l \in \{r+1, \ldots, r+s\}, \\
    \varepsilon_l(\overline{z_j}) &= -\delta_{l j} i, & l \in \{r+s+1, \ldots, r+2s\}.
\end{align*}

These functions form a basis of $A$. Each polynomial $\varepsilon_l$ lies in $A$, and the functions are linearly independent: at each point in $\mathfrak{V}(I)$, at most one function evaluates nonzero, except in the complex case where $\varepsilon_{r+l}(z_j)$ and $\varepsilon_{r+s+l}(z_j)$ are both nonzero. However, the system

\[
\begin{cases}
c_1 \varepsilon_{r+l}(z_j) + c_2 \varepsilon_{r+s+l}(z_j) = 0 \\
c_1 \varepsilon_{r+l}(\overline{z_j}) + c_2 \varepsilon_{r+s+l}(\overline{z_j}) = 0
\end{cases}
\quad \Longleftrightarrow \quad
\begin{cases}
c_1 + c_2 i = 0 \\
c_1 - c_2 i = 0
\end{cases}
\]
has only the trivial solution, confirming linear independence.

Every element of $A$ is a linear combination of the $\varepsilon_l$, since at each real root it is determined uniquely by one of $\varepsilon_1, \ldots, \varepsilon_r$, and at each complex root (along with its conjugate), it is determined by a linear combination of $\varepsilon_{r+l}$ and $\varepsilon_{r+s+l}$, with coefficients given by the real and imaginary parts of the target value.

Now, we analyze the bilinear map $H$ by examining $\varphi_{\varepsilon_l \varepsilon_p}$ for $l, p \in \{1, \ldots, r+2s\}$. The following cases arise:
\begin{enumerate}
    \item If $l = p \in \{1, \ldots, r\}$, then $\varphi_{\varepsilon_l \varepsilon_l}(\varepsilon_l) = \varepsilon_l$, and zero elsewhere. Hence, $\operatorname{Tr}(\varphi_{\varepsilon_l \varepsilon_l}) = 1$.
    
    \item If $l = p \in \{r+1, \ldots, r+s\}$, then $\varphi_{\varepsilon_l \varepsilon_l}(\varepsilon_l) = \varepsilon_l$, $\varphi_{\varepsilon_l \varepsilon_l}(\varepsilon_{l+s}) = \varepsilon_{l+s}$, and zero elsewhere, so $\operatorname{Tr}(\varphi_{\varepsilon_l \varepsilon_l}) = 2$.
    
    \item If $l = p \in \{r+s+1, \ldots, r+2s\}$, then $\varphi_{\varepsilon_l \varepsilon_l}(\varepsilon_l) = -\varepsilon_l$, $\varphi_{\varepsilon_l \varepsilon_l}(\varepsilon_{l-s}) = -\varepsilon_{l-s}$, due to $i^2 = -1$, and zero elsewhere. Thus, $\operatorname{Tr}(\varphi_{\varepsilon_l \varepsilon_l}) = -2$.
    
    \item If $l = p + s \in \{r+s+1, \ldots, r+2s\}$, then $\varphi_{\varepsilon_l \varepsilon_{l-s}}$ acts as follows:
    \[
    \varphi_{\varepsilon_l \varepsilon_{l-s}}(\varepsilon_l) = \varepsilon_l \varepsilon_{l-s} \varepsilon_l = \varepsilon_l^2 \varepsilon_{l-s} = \varepsilon_{l-s}, 
    \]
    \[
    \varphi_{\varepsilon_l \varepsilon_{l-s}}(\varepsilon_{l-s}) = \varepsilon_l \varepsilon_{l-s}^2 = -\varepsilon_l,
    \]
    but neither vector maps to itself, and it is zero elsewhere, so $\operatorname{Tr}(\varphi_{\varepsilon_l \varepsilon_{l-s}}) = 0$.
    
    \item The case $l + s = p \in \{r+s+1, \ldots, r+2s\}$ is symmetric to case 4 and yields trace zero.
    
    \item In all other cases, $\varphi_{\varepsilon_l \varepsilon_p} \equiv 0$, hence $\operatorname{Tr}(\varphi_{\varepsilon_l \varepsilon_p}) = 0$.
\end{enumerate}

The Gram matrix corresponding to $H$ in the basis $\{\varepsilon_l\}$ is diagonal:
\[
H = \operatorname{diag}(\underbrace{1, \dots, 1}_{r\text{ times}}, 
                         \underbrace{2, \dots, 2}_{s\text{ times}}, 
                         \underbrace{-2, \dots, -2}_{s\text{ times}}).
\]

Here, each entry corresponds to a root: 1 for each real root, 2 for each complex root, and -2 for each conjugate. Thus the rank is $r+2s$, or the total number of distinct complex roots, and the signature is $r+2s-2s=r$, or the number of distinct real roots.

\vspace{1em}

\subsection{Alternative proof under the assumption}~

Alternatively, one can make similar arguments under the assumption 
\[
A = \{g : \mathfrak{V}(I) \to \mathbb{C} \, | \, g \in \mathbb{R}[X_1, \ldots, X_n] \} 
\]
using a decomposition of A. As before, we build the basis of $\varepsilon_l, l\in \{ 1, \ldots r+2s\}$. Then, as these are all simply polynomials, we can conclude that 
\[
A \cong A_1 \times \cdots \times A_{r+s}
\]
where 
\begin{align*}
    A_j &= \{ g:\{x_j \} \to \mathbb{C} \, | \, g \in  \mathbb{R}[X_1, \ldots, X_n] \}, \quad j \in \{ 1, \ldots r\} \\
    A_{r+j} &= \{ g:\{z_j, \overline{z_j} \} \to \mathbb{C} \, | \, g \in  \mathbb{R}[X_1, \ldots, X_n] \}, \quad j \in \{ 1, \ldots s\}
\end{align*}
This isomorphism is true, as in the points $\{ x_1, \ldots x_r, z_1, \ldots z_s \}$ the values of the polynomials are independent and uniquely determined in the points $x_i$ and in each pair of points  $(z_j, \overline{z_j})$. As they are independent, we can conclude that 
\begin{align*}
    \operatorname{Rank}(H) &= \sum_{j=1}^{r+s} \operatorname{Rank}(H_j) \\
    \operatorname{Tr}(H) &= \sum_{j=1}^{r+s} \operatorname{Tr}(H_j)
\end{align*}
where $H_j$ are Hermite matrices built over $A_j$. For $j \in \{ 1, \ldots r \}$ these matrices are just a single number, namely, a 1, hence
\[
\operatorname{Rank}(H_j) = \operatorname{Tr}(H_j) = 1, \, j \in \{ 1, \ldots r \}
\]
For $j \in \{ r+1, \ldots r+s \}$ we observe that with our basis for any $l \in \{ r+1, \ldots r+s\}$
\begin{align*}
    \varepsilon_{l}\varepsilon_{l}\varepsilon_{l} &= \varepsilon_{l} \\
    \varepsilon_{l}\varepsilon_{l}\varepsilon_{l+s} &= \varepsilon_{l+s} \\
    \varepsilon_{l}\varepsilon_{l+s}\varepsilon_{l+s} &= -\varepsilon_{l} \\
    \varepsilon_{l+s}\varepsilon_{l+s}\varepsilon_{l+s} &= -\varepsilon_{l+s}
\end{align*}
which gives us the following matrices
\[
\begin{pmatrix}
1 & 0 \\
0 & 1
\end{pmatrix}
,
\begin{pmatrix}
-1 & 0 \\
0 & -1
\end{pmatrix}
,
\begin{pmatrix}
0 & -1 \\
-1 & 0
\end{pmatrix}
\]
for the maps $\varphi_{\varepsilon_{l} \varepsilon_{l}}$, $\varphi_{\varepsilon_{l+s} \varepsilon_{l+s}}$, $\varphi_{\varepsilon_{l+s} \varepsilon_{l}}$. These matrices commute. The according traces are $2$, $-2$ and $0$, therefore, we conclude that the matrix $H_j$ corresponding to these roots is 
\[
\begin{pmatrix}
    2 & 0 \\
    0 & -2
\end{pmatrix}
\]
These matrices have the $\operatorname{Rank}(H_j)=2$ and $\operatorname{Tr}(H_j)=0$, hence
\begin{align*}
    \operatorname{Rank}(H) &= \sum_{j=1}^{r+s} \operatorname{Rank}(H_j) = \sum_{j=1}^{r} \operatorname{Rank}(H_j) + \sum_{j=r+1}^{r+s} \operatorname{Rank}(H_j) = r+2s \\
    \operatorname{Tr}(H) &= \sum_{j=1}^{r+s} \operatorname{Tr}(H_j) = \sum_{j=1}^{r} \operatorname{Tr}(H_j) = r
\end{align*}

\vspace{1em}

\subsection{Reduction of the general case to the assumption}~

Now let us reduce the general case of the theorem to the special case. Recall that $I \triangleleft \mathbb{R}[X_1, \ldots, X_n]$ and $A = \mathbb{R}[X_1, \ldots, X_n]/I$. We want to study the rank and signature of 
\[
H : A \times A \to \mathbb{R}, \quad (f,g) \mapsto \operatorname{tr}(\varphi_{fg}),
\]
where $\varphi_f : A \to A, \, h \mapsto fh$. The best way to proceed is by considering a linear basis of $A$ and proving that the corresponding Gram matrix of $H$ has the same rank and signature as the Gram matrix of $H$ under our assumption 
\[
A = \{g : \mathfrak{V}(I) \to \mathbb{C} \, | \, g \in \mathbb{R}[X_1, \ldots, X_n]\}.
\]
For this, we also need to prove that $\dim(A) < \infty$, so that the notion of a basis is well-defined.

To that end, it is useful to observe $A/\sqrt{0}$. Let us use the following classic theorem:

\begin{theorem}[Hilbert's Nullstellensatz]
If $k$ is a field and $K/k$ is an algebraically closed field extension, for $I \triangleleft k[X_1, \ldots , X_n]$ an ideal in its polynomial ring, then with notation 
\[
\mathfrak{V}(I):=\{x\in K^n \, | \, \forall f \in I \, : f(x)=0 \} \\
\]
\[
\mathfrak{I}(W):=\{ f \in k[X_1, \ldots , X_n] \, | \, \forall x \in W \, : f(x) = 0\},\: W \subset K^n
\] 
\[
\sqrt{I} := \{ p \in k[X_1, \ldots , X_n] \, | \, \exists n \in \mathbb{N}: p^n \in I\}
\]
holds 
\[
\mathfrak{I}(\mathfrak{V}(I))=\sqrt{I}
\]
\end{theorem}
The proof of this theorem can be found in most textbooks on commutative algebra or algebraic geometry, e.g. \cite[Chapter 4, \S 1, Theorem~2, p.~179]{cox_little_oshea2015book}.

In our setting, we take \( k = \mathbb{R} \), \( K = \mathbb{C} \). 

In the ring \( A = \mathbb{R}[X_1, \ldots, X_n]/I \), the nilradical \( \sqrt{0} \) consists of classes \( p + I \) such that \( p^n \in I \) for some \( n \). Thus,
\[
\sqrt{0} = \sqrt{I}/I,
\]
and hence \( A/\sqrt{0} \cong \mathbb{R}[X_1, \ldots, X_n]/\sqrt{I} \) by the third isomorphism theorem.

By the Nullstellensatz, \( \sqrt{I} = \mathfrak{I}(\mathfrak{V}(I)) \). Therefore, two polynomials \( p_1, p_2 \in \mathbb{R}[X_1, \ldots, X_n] \) satisfy
\[
(p_1 - p_2) \in \sqrt{I} \quad \Leftrightarrow \quad p_1(x) = p_2(x) \text{ for all } x \in \mathfrak{V}(I),
\]
which shows that $p_1+\sqrt{I} = p_2+\sqrt{I} \Leftrightarrow p_1(x)=p_2(x), \quad \forall x \in \mathfrak{V}(I)$, which means that in the quotient \( \mathbb{R}[X_1, \ldots, X_n]/\sqrt{I} \), elements differ only by their values in \( \mathfrak{V}(I) \). In other words,
\[
A/\sqrt{0} \cong \mathbb{R}[X_1, \ldots, X_n]/\sqrt{I} \cong \{g : \mathfrak{V}(I) \to \mathbb{C} \mid g \in \mathbb{R}[X_1, \ldots, X_n] \}.
\]

This is exactly the setting of the special case we proved earlier. In particular, this shows that $\dim(A/\sqrt{0}) < \infty$ as a linear subspace. It remains to deduce that $\dim(A) < \infty$, in order to use a basis for the Gram matrix. To that end let us prove the following lemma:

\begin{lemma}
For any ideal $I \triangleleft \mathbb{R}[X_1, \ldots , X_n]$ and $A = \mathbb{R}[X_1, \ldots , X_n]/I$ we have
\[
\dim(A/\sqrt{0})<\infty \Leftrightarrow \dim(A)<\infty
\]
as vector spaces, where 
\begin{align*}
    \sqrt{0}&:=\{a \in A \, | \, \exists n \in \mathbb{N}: a^n = 0\} \\
    &= \{ p+I \, | \, p \in \mathbb{R}[X_1, \ldots , X_n], \exists n \in \mathbb{N}: \, p^n \in I\} = \sqrt{I}/I
\end{align*}
is the nilradical over $A$.
\end{lemma}

\begin{proof}
As $\sqrt{0}$ is a linear subspace of $A$, let us consider the linear decomposition of $A$. Let $U$ be the linear complement of $\sqrt{0}$, i.e.~$A = \sqrt{0} \oplus U$. 

Note that $U \cong A/\sqrt{0}$ as linear subspaces, under the residue isomorphism. It follows that $\dim(A) < \infty$ if and only if both $\dim(A/\sqrt{0}) < \infty$ and $\dim(\sqrt{0}) < \infty$. Thus, it suffices to prove that $\dim(\sqrt{0}) < \infty$ under the assumption $\dim(A/\sqrt{0}) < \infty$.

Since $A$ is a finitely generated algebra over a field, it is Noetherian. Therefore, the ideal $\sqrt{0} \triangleleft A$ is finitely generated as an ideal.

Let $\sqrt{0} = \langle g_1, \ldots, g_k\rangle$. Let $N_j \in \mathbb{N}$ be the smallest natural numbers such that $g_j^{N_j} = 0$ for all $j \in \{1, \ldots, k\}$, and let $N = \max\{N_1, \ldots, N_k\}$. Then for all $h \in \sqrt{0}$ we have $h = \sum_{j=1}^k h_j g_j$ with $h_j \in A$, and
\[
h^M = \sum_{|\alpha|=M} \widehat{h_\alpha} g^\alpha, \quad \forall M \in \mathbb{N},
\]
where $\alpha$ is a non-negative multiindex. This means however, that for $M \geq k(N-1)+1$ we have $h^M = 0$ by the pigeonhole principle.

Now for every $h \in \sqrt{0}$ with $h = \sum_{j=1}^k h_j g_j$, $h_j \in A$, we can decompose
\[
\forall j \in \{1, \ldots, k\}: \quad h_j \in A \;\Rightarrow\; h_j = \widetilde{h_j} + h_j', \quad
\widetilde{h_j} \in U, \quad h_j' \in \sqrt{0}.
\]
We can now take each of the $h_j'$ as a new $h$ and again decompose it over our ideal. Continuing this process, after at most $M$ decompositions we reach products of $g_j$ of total degree $M$, i.e.~a point where our original $h$ is decomposed as 
\[
h = \sum_{|\alpha| < M} \widehat{h_\alpha} g^{\alpha}, \quad \widehat{h_\alpha} \in U,
\]
since all terms in $\sqrt{0}$ will satisfy $|\alpha| \geq M$. This sum has finitely many terms. We are working under the assumption that $A/\sqrt{0}$, and hence $U$, is finite-dimensional. Therefore, decomposing the $\widehat{h_\alpha}$ in $U$ yields a finite linear combination, and thus $\dim(\sqrt{0}) < \infty$.
\end{proof}

We now return to the proof of the multivariate Hermite criterion. Since $\linebreak\dim(A/\sqrt{0}) < \infty$, the lemma above implies that $\dim(A) < \infty$. We can now construct a basis $B$ of $\sqrt{0}$ and extend it to a basis $E$ of $A$. 

We construct the Gram matrix of $H : A \times A \to \mathbb{R}, \quad (f,g) \mapsto \operatorname{tr}(\varphi_{fg})$. Any basis element $g \in \sqrt{0}$ is nilpotent. Therefore, for any $f \in A$ the product $fg$ is also nilpotent, which means $\operatorname{tr}(\varphi_{fg}) = 0$, as all its eigenvalues are 0. Hence, all matrix entries involving those basis elements will be zero. The matrix therefore takes the block form:

\[
\left(
\begin{array}{c|c}
\begin{matrix}
\Phi_B (H)
\end{matrix} & \begin{matrix}
0
\end{matrix} \\ \hline
\begin{matrix}
0
\end{matrix} & \begin{matrix}
0
\end{matrix}
\end{array}
\right).
\]

Here, \( \Phi_B(H) \) is the Gram matrix associated to the basis of $U$, which is the linear complement of $\sqrt{0}$ and is $U \cong A/\sqrt{0}$. This shows that we may restrict our attention to $A/\sqrt{0} \cong \{g : \mathfrak{V}(I) \to \mathbb{C} \mid g \in \mathbb{R}[X_1, \ldots, X_n] \}$.

Hence the general multivariate Hermite criterion follows.

\section{Implementation}
Let us now apply the multivariate Hermite criterion to determine the number of real and complex roots of a system of polynomial equations. This requires knowledge of Gröbner bases and the ability to perform reduction modulo an ideal \( I \). For a brief overview or refresher, see \cite[Chapter~2]{cox_little_oshea2015book}.

Recall that we have a family of polynomials \(f_1, \ldots, f_k \in \mathbb{R}[X_1, \ldots, X_n]\) with a finite (zero-dimensional) set of zeros, the generated ideal \(I = \langle f_1, \ldots, f_k \rangle\), and the quotient ring \(A = \mathbb{R}[X_1, \ldots, X_n]/I\). The goal is to determine the number of distinct real and complex roots. To that end, we define the linear maps \(\varphi_f : A \to A, \, h \mapsto fh\), and the bilinear form \(H : A \times A \to \mathbb{R}, \, (f, g) \mapsto \operatorname{tr}(\varphi_{fg})\). Then, the number of distinct complex roots is \(\operatorname{Rank}(H)\), and the number of distinct real roots is \(\operatorname{Tr}(H)\).
 
We proceed according to the algorithm outlined below. 

To compute \( \operatorname{tr}(\varphi_{fg}) \), we need to represent the multiplication operator \( \varphi_{fg} \) as a matrix \( M_B^B(\varphi_{fg}) \) with respect to some linear basis \( B \) of the quotient ring \( A \).

The most straightforward and reliable way to construct such a basis is to generate all monomials up to the maximum degree appearing in the generators \( f_1, \ldots, f_k \). These can be formed by expanding \( (1 + x_1 + \cdots + x_n)^d \), where \( d \) is the maximum total degree among the \( f_i \). We then reduce this list of monomials modulo the ideal \( I \), using the Gröbner basis of \( I \). Decomposing all remainders into monomials and eliminating duplicates produces a basis of \( A \).

Once we have the basis, we compute the products \( fg \) for each pair \( f, g \in B \), reduce these products modulo \( I \), and express the results in terms of the basis. These expressions yield the matrix representations \( M_B^B(\varphi_{fg}) \). Computing the trace of each such matrix gives us the entries of the Hermite matrix \( H \).

Here is a Macaulay2 implementation of this algorithm:
\begin{lstlisting}[language=Macaulay2]
kk=QQ;

-- Function to normalize monomials for later
normalizeMonomial = m -> (
	if leadCoefficient(m)==0 then m
	else m/leadCoefficient(m)
);

-- signature function for the end of the code
signature = M -> (
    evs = eigenvalues(M, Hermitian=>true);
    p   = #select(evs, lambda -> lambda > 0);
    q   = #select(evs, lambda -> lambda < 0);
    p - q
);

myHermite = args -> (
	
	-- Creation of the polynomial ring and the ideal
	n = args_0;
	var = apply(toList(1..n), i -> concatenate("x", toString(i)));
	R = QQ[var];
	var = flatten entries vars R;

	pols = args_{1..#args-1};
	pols = apply(pols, s -> value s);
	I = ideal(pols);

	-- Check if the ideal is zero dimensional
	if dim I > 0 then error "The ideal is not zero-dimensional";
	
	-- Creation of the Linear Basis of A = R/I
	
	MaxMonDeg = degree max apply(pols, p -> normalizeMonomial(p));
	F = sum var;
	MonList = flatten entries monomials((1+F)^(MaxMonDeg_0), Variables=> var);
	ReducedMonList = apply(MonList, p -> p%I);
	LinearBasis = flatten apply(ReducedMonList, p -> terms p);
	LinearBasis = apply(LinearBasis, p -> normalizeMonomial(p));
	LinearBasis = sort(toList set LinearBasis);
	
	-- Application of the Hermite Criterion
	n = length(LinearBasis);
	H = mutableMatrix table(n, n, (i,j) -> 0_R);

	M = mutableMatrix table(n, n, (k,l) -> 0_R);

	variables = flatten entries vars R;

	for i in 0..n-1 do (
		for j in 0..n-1 do (
			f = LinearBasis_i;
			g = LinearBasis_j;
			for k in 0..n-1 do (
				(SkipMat, Coeff) = coefficients(((f*g*LinearBasis_k)%I), Variables => variables, Monomials => LinearBasis);
				for l in 0..n-1 do (
					M_(k,l) = Coeff_0_l;
				);
			);
			N = matrix M; 
			H_(i,j)= trace N;
		);
	);
	print("Hermite matrix:");
	print(H);
	
	print("number of complex solutions:");
	print(rank H);
	
	print("number of real solutions:");
	immutH = matrix H;
	HRR = substitute(immutH, RR);
	print(signature(HRR));
);

\end{lstlisting}

To apply this code to a particular system, one simply runs a command of the following form:
\begin{lstlisting}[language = Macaulay2]
myHermite(2, "x1*x2+x2-1", "x1^2+x2^2-1")
\end{lstlisting}

\noindent Output:
\begin{lstlisting}
Hermite matrix:
| 4  0 -2 0  |
| 0  0 4  6  |
| -2 4 4  0  |
| 0  6 0  -4 |
number of complex solutions:
4
number of real solutions:
2
\end{lstlisting}

This implementation is compatible with any number of variables (which has to be specified in the first variable of the call and have to be called \( x1, x2, \ldots \) ).

For example, here is the case where the Hermite matrix is not of full rank:
\begin{lstlisting}[language = Macaulay2]
myHermite(2, "x1-1", "x1^2+x2^2-1")
\end{lstlisting}
\begin{lstlisting}[language = Macaulay2]
Hermite matrix:
| 2 0 |
| 0 0 |
number of complex solutions:
1
number of real solutions:
1
\end{lstlisting}

Alternatively, the Macaulay2 package \textbf{RealRoots} provides the functions \texttt{traceCount} and \texttt{realCount}, which compute the number of distinct complex and real solutions, respectively. A detailed description of this package can be found in \cite{garcia_maluccio_sottile_yahl2024article}.

\vspace{1em}

\subsection{Speed}~

The most computationally intensive part of this algorithm is the calculation of the Gröbner basis of $I$ for further reductions modulo $I$. Fortunately, Macaulay2 is optimized for such operations. The complexity of this algorithm depends on the highest power $d$ present in $f_i$ and the number of variables $n$, and is at worst $O(d^{2^n})$. This means the runtime grows polynomially in $d$, and exponentially in $n$:

\begin{lstlisting}[language = Macaulay2]
timing myHermite(2, "x1-x2", "x1^2-x2") -- .0106531 seconds
timing myHermite(2, "x1-x2", "x1^3-x2") -- .0308254 seconds
timing myHermite(2, "x1-x2", "x1^4-x2") -- .0417585 seconds
timing myHermite(2, "x1-x2", "x1^5-x2") -- .0844115 seconds
timing myHermite(2, "x1-x2", "x1^6-x2") -- .127797 seconds
timing myHermite(2, "x1-x2", "x1^7-x2") -- .228504 seconds
\end{lstlisting}

\begin{lstlisting}[language = Macaulay2]
timing myHermite(2, "x1-1", "x1^2+x2^2-1") -- .0277464 seconds
timing myHermite(3, "x1-1", "x1^2+x2^2-1", 
				"x1^2+x2^2+x3^2-1") -- .107985 seconds
timing myHermite(4, "x1-1", "x1^2+x2^2-1", 
				"x1^2+x2^2+x3^2-1", 
				"x1^2+x2^2+x3^2+x4^2-1") -- .209921 seconds
timing myHermite(5, "x1-1", "x1^2+x2^2-1", 
				"x1^2+x2^2+x3^2-1", 
				"x1^2+x2^2+x3^2+x4^2-1", 
				"x1^2+x2^2+x3^2+x4^2+x5^2-1") -- 1.1736 seconds
timing myHermite(6, "x1-1", "x1^2+x2^2-1", 
				"x1^2+x2^2+x3^2-1", 
				"x1^2+x2^2+x3^2+x4^2-1", 
				"x1^2+x2^2+x3^2+x4^2+x5^2-1", 
				"x1^2+x2^2+x3^2+x4^2+x5^2+x6^2-1") 
                -- 4.30042 seconds
timing myHermite(7, "x1-1", "x1^2+x2^2-1", 
				"x1^2+x2^2+x3^2-1", 
				"x1^2+x2^2+x3^2+x4^2-1", 
				"x1^2+x2^2+x3^2+x4^2+x5^2-1", 
				"x1^2+x2^2+x3^2+x4^2+x5^2+x6^2-1", 
				"x1^2+x2^2+x3^2+x4^2+x5^2+x6^2+x7^2-1") 
                -- 14.5534 seconds
\end{lstlisting}

\bibliographystyle{plain} 
\bibliography{references} 

@unpublished{plaumann2023seminar,
  author = {Daniel Plaumann},
  title = {Real Algebraic Geometry},
  note = {Tutorial lectures, Institut Henri Poincaré},
  year = {2023},
  url = {https://www.mathematik.tu-dortmund.de/~dplauman/notes/rag-ihp.pdf}
}

@book{basu_pollack_roy2003book,
  author = {Saugata Basu and Richard Pollack and Marie-Françoise Roy},
  title = {Algorithms in Real AlgebraicGeometry}, 
  publisher = {Springer},    
  year = {2003}  
}

@book{cox_little_oshea2015book,
  author = {David A. Cox and John Little and Donal O'Shea},
  title = {Ideals, Varieties, and Algorithms}, 
  publisher = {Springer Nature},    
  year = {2015}  
}

@article{garcia_maluccio_sottile_yahl2024article,
    author = {Jordy Lopez Garcia and Kelly Maluccio and Frank Sottile and Thomas Yahl},
    title = {Real solutions to systems of polynomial equations in Macaulay2},
    journal = {MSP},
    year = {2024}
}

\end{document}